\documentclass[11pt]{article}
\usepackage{amsmath,amsthm,amssymb,hyperref,mathtools}
\usepackage[left=1in,top=1in,right=1in]{geometry}

\theoremstyle{definition}
\newtheorem{theorem}{Theorem}[section]
\newtheorem{lemma}[theorem]{Lemma}

\counterwithin{equation}{section}
\allowdisplaybreaks

\newcommand{\vol}{\text{vol}}

\title{On average hitting time and Kemeny's constant for weighted trees}

\author{Ji Zeng\thanks{Department of Mathematics, University of California at San Diego, La Jolla, CA, 92093 USA. Partly supported by NSF grant DMS-1800746. Email: {\tt jzeng@ucsd.edu}.} }
\date{}

\begin{document}

\maketitle

\begin{abstract}
For a connected graph $G$, the average hitting time $\alpha(G)$ and the Kemeny's constant $\kappa(G)$ are two similar quantities, both measuring the time for the random walk on $G$ to travel between two randomly chosen vertices. We prove that, among all weighted trees whose edge weights form a fixed multiset, $\alpha$ is maximized by a special type of ``polarized'' paths and is minimized by a unique weighted star graph. We also obtain a similar characterization of the $\kappa$-maximizing and $\kappa$-minimizing elements among such a collection of weighted trees. Our proofs are based on the forest formulas for $\alpha$ and $\kappa$.
\end{abstract}

\section{Introduction}\label{section_intro}

In this paper, we consider random walks on graphs whose edges are assigned positive weights. When the random walk on a connected weighted graph $G=(V,E)$ is at vertex $u$, it will go to vertex $v$ with probability proportional to the weight of the edge $uv$ (see Section~1.4 in \cite{lyons2017probability}). The \emph{hitting time} $H_G(u,v)$ is the expected number of steps to walk from vertex $u$ to vertex $v$. We note that $H_G(v,v) = 0$ trivially for all vertex $v$. We shall study the \emph{average hitting time}, defined as the arithmetic mean of all hitting times of $G$,
    \[\alpha(G):=\frac{1}{|V|^2}\sum_{u\in V}\sum_{v\in V} H_G(u,v).\]

The \emph{stationary distribution} of $G$ is a unique probability vector $\pi_G$ satisfying $\pi_G^T=\pi_G^T P_G$, where $P_G$ is the transition probability matrix for the random walk on $G$ (see Exercise~2.1 in \cite{lyons2017probability}). The \emph{Kemeny's constant} of $G$, similar to $\alpha(G)$ but more complicated, is defined as
    \[\kappa(G):=\sum_{v\in V}H_G(u,v)\pi_G(v).\]
In words, $\kappa(G)$ is the expected time to travel from a fixed starting vertex $u$ to a random destination sampled from the stationary distribution. Formally, this definition of $\kappa(G)$ depends on the choice of $u$, but Kemeny and Snell \cite{kemeny} observed the surprising fact that the value $\kappa(G)$ is independent of $u$ (see Exercise~2.48 in \cite{lyons2017probability}). The Kemeny's constant has applications in network theory~\cite{levene2002kemeny}, robotics~\cite{patel2015robotic} and chemistry~\cite{li2019multiplicative}.

The original motivation of this paper is to prove the following.
\begin{theorem}\label{minmax_unweighted}
Among all unweighted trees of a fixed size, the path has the largest $\alpha$ and $\kappa$, and the star graph has the smallest $\alpha$ and $\kappa$.
\end{theorem}

We note that, among all unweighted trees of a fixed size, Ciardo et al. \cite{ciardo2022kemeny} previously proved the star graph minimizes $\kappa$; and recently, Faught et al. \cite{faught20221} proved the path maximizes $\kappa$. Our approach will be somewhat different from theirs, and it generalizes to a weighted setting: For a multiset (i.e. set with multiplicities) $W$ of positive numbers, we let $\mathcal{T}_W$ be the collection of all weighted trees (up to isomorphism) $T$ such that $W$ equals $\{\omega(e)| e\in E(T)\}$ as a multiset. Here, $\omega(e)$ denotes the weight of the edge $e$. See Figure~\ref{fig:TW} for examples of elements in $\mathcal{T}_W$. Notice that when $W$ only consists of multiple $1$'s, $\mathcal{T}_W$ is the collection of all unweighted trees of a fixed size. We will investigate which elements in $\mathcal{T}_W$ maximize or minimize $\alpha$ and $\kappa$ for a general $W$.

Clearly, there is a unique element $S_W$ in $\mathcal{T}_W$ whose underlying graph structure is a star graph. On the other hand, there are many weighted paths in $\mathcal{T}_W$ due to various ways to assign the edge weights. We care about a special type of ``polarized'' paths: For a weighted path $P=v_1v_2\dots v_n$, we denote its edges as $e_i:=v_{i}v_{i+1}$ and define $c(e_i):=\min\{i,n-i\}$ for $i=1,\dots,n-1$, then we say $P$ is \emph{polarized} provided $\omega(e_i) \geq \omega(e_j)$ whenever $c(e_i) \leq c(e_j)$.

Intuitively, $c(e_i)$ measures how centrally $e_i$ is located in $P$: We regard $e_1$ and $e_{n-1}$, the edges containing the leaves of $P$, as the least central edges hence they have the smallest $c(e_1)=c(e_{n-1})=1$; We regard $e_{2}$ and $e_{n-2}$ as the second least central edges hence they have the second smallest $c(e_2)=c(e_{n-2})=2$ and so on; We regard the middle edge $e_{\lceil\frac{n-1}{2} \rceil}$ (as well as $e_{\lceil\frac{n-1}{2} \rceil + 1}$ if $n$ is odd) as the most central edge hence it has the largest $c\left(e_{\lceil\frac{n-1}{2} \rceil}\right)=\lceil\frac{n-1}{2} \rceil$. So, in words, a weighted path is polarized if edges located more centrally are assigned smaller weights. See Figure~\ref{fig:TW} for an illustration.

\begin{figure}[ht]
    \centering
    \includegraphics[scale=1.1]{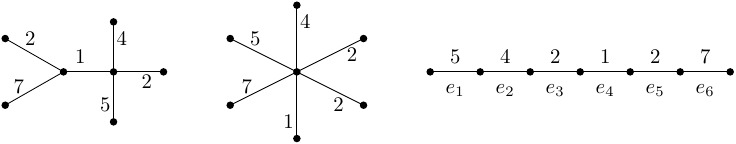}
    \caption{From left to right: a weighted tree in $\mathcal{T}_W$ where $W=\{7,5,4,2,2,1\}$ and the numbers indicate edge weights; the weighted star $S_W$; a polarized path in $\mathcal{T}_W$ where $c(e_1)=c(e_6)=1$, $c(e_2)=c(e_5)=2$, and $c(e_3)=c(e_4)=3$. }
    \label{fig:TW}
\end{figure}

\begin{theorem}\label{minmax_aht}
In $\mathcal{T}_W$, where $W$ is a multiset of positive numbers, the polarized paths are the elements that maximize $\alpha$; and $S_W$ is the unique element that minimizes $\alpha$.
\end{theorem}

By a slight modification of the proof of Theorem~\ref{minmax_aht}, we obtain a similar result for the Kemeny's constant.
\begin{theorem}\label{minmax_kemeny}
In $\mathcal{T}_W$, where $W$ is a multiset of positive numbers, the elements that maximize $\kappa$ are exactly the paths $P=v_1v_2\dots v_n$ whose edge weight assignment maximizes
\[\sum_{i=1}^{n-1}\sum_{j=1}^{i-1}\sum_{k=i+1}^{n-1} \frac{\omega(e_j)\omega(e_k)}{\omega(e_i)},\]
where $e_i:=v_iv_{i+1}$ for $i=1,\dots,n-1$; and $S_W$ is the unique element that minimizes $\kappa$.
\end{theorem}

The characterization in Theorem~\ref{minmax_kemeny} is not as clean as that in Theorem~\ref{minmax_aht}. The major contribution in the proof of Theorem~\ref{minmax_kemeny} is to show that a $\kappa$-maximizing element in $\mathcal{T}_W$ has to be a path. Nevertheless, there is only one unweighted path with a given size, so Theorem~\ref{minmax_unweighted} is implied by Theorems~\ref{minmax_aht}~and~\ref{minmax_kemeny}.

Inside a general $\mathcal{T}_W$, the $\alpha$-maximizing elements are quite different from the $\kappa$-maximizing elements. For example, when $W=\{10,8,1,1,0.1\}$, we can check through Theorem~\ref{minmax_kemeny} that $\kappa$ is maximized by a path with edge weight assignment $\omega_1=10$, $\omega_2=0.1$, $\omega_3=1$, $\omega_4=1$ and $\omega_5=8$, which is not polarized, hence not $\alpha$-maximizing by Theorem~\ref{minmax_aht}.

The rest of this paper is organized as follows: Section~\ref{section_pre} presents definitions, notations, and the forest formulas for $\alpha$ and $\kappa$; Sections~\ref{section_aht}~and~\ref{section_kemeny} are devoted to the proof of Theorems~\ref{minmax_aht}~and~\ref{minmax_kemeny} respectively; Section~\ref{section_remark} lists some remarks.

\section{Preliminaries}\label{section_pre}

Let $G=(V,E)$ be a connected weighted graph without loops. For each edge $e=uv\in E$, we denote its \emph{weight} by $\omega(e)=\omega(uv)>0$. It is required that $\omega(e)=0$ for $e \in V\times V\setminus E$. For any subgraph $H$ of $G$, we define its weight by
   \[\omega(H)=\prod_{e\in E(H)}\omega(e).\]

For each vertex $u$, we define its \emph{degree} by $d_G(u):=\sum_{v\in V} \omega(uv)$. The \emph{volume} of a subset $U\subset V$ refers to the quantity $\vol_G(U):=\sum_{u\in U}d_G(u)$. The volume of a subgraph $H\subset G$ refers to the volume of $V(H)\subset V$ and is denoted by $\vol_{G}(H)$. We remind our readers that in general $\vol_G(H)\neq \vol_H(H)$. The \emph{size} of $G$ refers to the number of its vertices and is denoted by $|G|$.

A \emph{tree} is a graph that is acyclic and connected. A \emph{$2$-forest} is a graph that is acyclic and has two connected components. A subgraph $H\subset G$ is said to be \emph{spanning} if $V(H)=V(G)$. We use $\mathbb{F}_1(G)$ to denote the set of all spanning sub-trees of $G$, and use $\mathbb{F}_2(G)$ to denote the set of all spanning sub-$2$-forests of $G$.

We shall use the following two lemmas in our proofs. They express $\alpha(G)$ and $\kappa(G)$ as weighted enumerations of elements in $\mathbb{F}_1(G)$ and $\mathbb{F}_2(G)$. For simplicity of presentation, we define the quantity 
   \[\tau(G):=\sum_{T\in \mathbb{F}_1(G)}\omega(T);\]
And we define for any $2$-forest $F\in \mathbb{F}_2(G)$, whose two components are denoted as $T_1$ and $T_2$,
  \[S(F):=|T_1||T_2|\qquad\text{and}\qquad V_G(F):=\vol_{G}(T_1)\vol_G(T_2).\]

\begin{lemma}\label{forestformula_aht}
For a connected weighted graph $G$, we have
   \[\frac{|G|}{\vol_G(G)}\alpha(G)=\frac{1}{|G|\tau(G)}\sum_{F\in \mathbb{F}_2(G)}S(F)\omega(F).\]
\end{lemma}

\begin{lemma}\label{forestformula_kemeny}
For a connected weighted graph $G$, we have
   \[\kappa(G)=\frac{1}{\vol_G(G)\tau(G)}\sum_{F\in \mathbb{F}_2(G)}V_G(F)\omega(F).\]
\end{lemma}

Lemma~\ref{forestformula_aht} is established in \cite{chung2021forest} as Corollary~4.5. Lemma~\ref{forestformula_kemeny} is established in \cite{chebotarev2007graph} (see also \cite{chebotarev2020hitting} and \cite{kirkland2016kemeny}) and our statement here takes the form of Corollary~4.4(ii) in \cite{chung2021forest}.

\section{Proof of Theorem~\ref{minmax_aht}}\label{section_aht}

We consider the following operation on a tree $T$: Take two adjacent edges $e_1=v_1v_2$ and $e_2=v_2v_3$, and name the component of $T\setminus\{e_1,e_2\}$ containing $v_i$ as $T_i$, for $i=1,2,3$. If $|T_1|>|T_2|$, we construct another tree $T'$ from $T$ by deleting the edge $e_2$ and adding an edge $e_2'=v_1v_3$ with weight $\omega(e_2'):=\omega(e_2)$. In this case, we say \emph{$T$ edge-transfers to $T'$ with respect to size}. See Figure~\ref{fig:edge_transfer} for an illustration.
\begin{figure}[ht]
    \centering
    \includegraphics[scale=1.1]{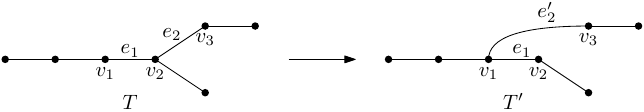}
    \caption{$T$ edge-transfers to $T'$ with respect to size.}
    \label{fig:edge_transfer}
\end{figure}

\begin{lemma}\label{inequality_aht}
If $T$ edge-transfers to $T'$ with respect to size, then $\alpha(T)>\alpha(T')$.
\end{lemma}
\begin{proof}
We keep the notations as in the previous paragraph. Since $T$ and $T'$ are trees, by definition of $\tau(T)$ and $\tau(T')$, we have $\tau(T)=\omega(T)=\omega(T')=\tau(T')=:\tau$. It is also obvious that $|T|=|T'|=:n$ and $\vol_T(T)=\vol_{T'}(T')=:\vol$.

Notice that for any edge $e\not\in \{e_1,e_2,e_2'\}$, the $2$-forests $T\setminus e$ and $T'\setminus e$ give the same component partition of vertices, that is, two vertices $x,y$ are in the same component of $T\setminus e$ if and only if they are in the same component of $T'\setminus e$. See Figure~\ref{fig:proof} for an illustration. So we have $S(T\setminus e)=S(T'\setminus e)$. We also have $\omega(T\setminus e)=\tau/\omega(e)=\omega(T'\setminus e)$. The two $2$-forests $T\setminus e_2$ and $T'\setminus e_2'$ are equal as graphs, so $S(T\setminus e_2)=S(T'\setminus e'_2)$ and $\omega(T\setminus e_2)=\omega(T'\setminus e'_2)=\tau/\omega(e_2)$. Therefore, the difference between $\alpha(T)$ and $\alpha(T')$ in the formula of Lemma~\ref{forestformula_aht} is only contributed by the $2$-forests $T\setminus e_1$ and $T'\setminus e_1$:\begin{equation}\label{eq:inequality_aht_1}
    \frac{n}{\vol}(\alpha(T)-\alpha(T'))=\frac{1}{n\omega(e_1)}(S(T\setminus e_1)-S(T'\setminus e_1)).
\end{equation} Here we used $\omega(T\setminus e_1)=\tau/\omega(e_1)=\omega(T'\setminus e_1)$.
\begin{figure}[ht]
    \centering
    \includegraphics[scale=1.1]{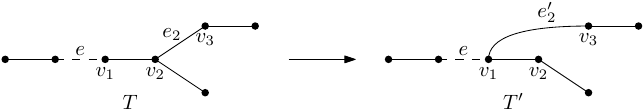}
    \caption{$T\setminus e$ and $T'\setminus e$ give the same component partition of vertices.}
    \label{fig:proof}
\end{figure}

The two components of $T\setminus e_1$ are $T_1$ and $T_2\cup T_3$ respectively, and the two components of $T'\setminus e_1$ are $T_1\cup T_3$ and $T_2$ respectively. Hence we have
    \[S(T\setminus e_1)-S(T'\setminus e_1)= |T_1|(|T_2|+|T_3|)-(|T_1|+|T_3|)|T_2|=(|T_1|-|T_2|)|T_3|>0,\]
since $|T_1|>|T_2|$ by our assumption of the edge-transfer. This inequality together with \eqref{eq:inequality_aht_1} concludes $\alpha(T)>\alpha(T')$.
\end{proof}

Let us write $T\succeq_s T'$ if $T'$ can be obtained after applying multiple (possibly zero times) edge-transfers with respect to size starting with $T$. As a consequence of Lemma~\ref{inequality_aht}, the relation $\succeq_s$ defines a partial order on $\mathcal{T}_W$.

\begin{lemma}\label{poset_aht}
For any multiset $W$ of positive numbers, the maximal elements of the poset $(\mathcal{T}_W,\succeq_s)$ are the weighted paths, and the minimal element of $(\mathcal{T}_W,\succeq_s)$ is the weighted star $S_W$.
\end{lemma}
\begin{proof}
Suppose $T'\in \mathcal{T}_W$ is not a weighted path, then $T'$ has a vertex $v$ with at least three vertices $u_1,u_2,u_3$ adjacent to $v$. Let $T'_i$ be the component of $T'\setminus vu_i$ containing $u_i$, for $i=1,2,3$. Without loss of generality, we can assume $|T'_1|\geq |T'_2|$. Now, we construct another tree $T$ in $\mathcal{T}_W$ from $T'$ by deleting the edge $vu_3$ and adding the edge $u_2u_3$ with weight $\omega(u_2u_3):=\omega(vu_3)$. Then $T$ edge-transfers to $T'$ with respect to size. Indeed, in $T$, if we name $v_1:=v$, $v_2:=u_2$, $v_3:=u_3$, $e_1:=v_1v_2$, $e_{2}:=v_2v_3$ and the component of $T\setminus\{e_1,e_2\}$ containing $v_i$ as $T_i$, we have $|T_1|\geq |T'_1\cup \{v\}|>|T'_1|\geq |T'_2|=|T_2|$. So we can perform an edge-transfer on $T$, taking $e_1,e_2$ as the two adjacent edges, to obtain $T'$. By Lemma~\ref{inequality_aht}, $T$ and $T'$ are not isomorphic, so $T'$ is not a maximal element of $(\mathcal{T}_W,\succeq_s)$. On the other hand, since every vertex of a weighted path $P\in \mathcal{T}_W$ has at most two neighbors, $P$ cannot be a result of an edge-transfer, so $P$ is maximal with respect to $\succeq_s$.

Suppose $T\in \mathcal{T}_W$ is not $S_W$, then $T$ contains a path of length three which we can write as $v_0v_1v_2v_3$. Let $C_i$ be the component of $T\setminus v_1v_2$ containing $v_i$, for $i=1,2$. Without loss of generality, we can assume $|C_1|\geq |C_2|$. Then we can perform an edge-transfer with respect to size on $T$ to obtain another tree $T'$. Indeed, if we name $e_1:=v_1v_2$, $e_{2}:=v_2v_3$ and the component of $T\setminus\{e_1,e_2\}$ containing $v_i$ as $T_i$, we have $|T_1|=|C_1|\geq |C_2|>|C_2\setminus \{v_3\}|\geq |T_2|$. By Lemma~\ref{inequality_aht}, $T$ and $T'$ are not isomorphic, so $T$ is not a minimal element of $(\mathcal{T}_W,\succeq_s)$. On the other hand, it is easy to check that $S_W$ is minimal with respect to $\succeq_s$.
\end{proof}

\begin{proof}[Proof of Theorem~\ref{minmax_aht}]
As a consequence of Lemma~\ref{inequality_aht} and Lemma~\ref{poset_aht}, the elements that maximize $\alpha$ are among the weighted paths in $\mathcal{T}_W$, and $S_W$ is the unique element that minimizes $\alpha$.

Let $P=v_1v_2\dots v_n$, where $n=|W|+1$, be a weighted path and name the edges $e_i:=v_iv_{i+1}$ for $i=1,\dots,n-1$. According to Lemma~\ref{forestformula_aht}, we can compute
\[\alpha(P)=\frac{\vol_P(P)}{|P|^2\tau(P)}\sum_{i=1}^{n-1}S(P\setminus e_i)\omega(P\setminus e_i)=\frac{\sum_{w\in W} 2w}{n^2}\sum_{i=1}^{n-1}i(n-i)\omega(e_i)^{-1}.\]
In order to maximize $\alpha(P)$, we need to assign the edge weights such that $\omega(e_j)^{-1}\geq \omega(e_i)^{-1}$ if and only if $j(n-j)\geq i(n-i)$. It is easy to check that $j(n-j)\geq i(n-i)$ is equivalent to $\min\{j,n-j\}\geq \min\{i,n-i\}$ for $i,j=1,\dots,n-1$. Recall from Section~\ref{section_intro} that $c(e_i)=\min\{i,n-i\}$. Hence, a necessary condition for $P$ to be $\alpha$-maximizing is that $\omega(e_i)\geq \omega(e_j)$ whenever $c(e_i)\leq c(e_j)$, i.e. $P$ is polarized.

Now it suffices to show that all polarized paths in $\mathcal{T}_W$ have the same $\alpha$. To see this, we write $W=\{w_1,w_2,\dots, w_n\}$ such that $w_i\geq w_{i+1}$ for $i=1,\dots,n-1$. We also partition the edge set of $P$ into $E_i:=\{e_i,e_{n-i}\}$ for $i=1,\dots, \lceil \frac{n-1}{2} \rceil$. Since $e\in E_i$ if and only if $c(e)=i$, to generate all polarized paths in $\mathcal{T}_W$, it suffices to assign the weights $w_{2i-1}$ and $w_{2i}$ to edges in $E_i$ for all $i$. Because $S(P\setminus e_i)=S(P\setminus e_{n-i})$, the two possible assignments for $E_i$, $(\omega(e_i),\omega(e_{n-i}))=(w_{2i-1},w_{2i})$ or $(\omega(e_i),\omega(e_{n-i}))=(w_{2i},w_{2i-1})$, will not affect $\alpha$ in the formula of Lemma~\ref{forestformula_aht}. Hence, all weighted paths generated by this process have the same $\alpha$.
\end{proof}

\section{Proof of Theorem~\ref{minmax_kemeny}}\label{section_kemeny}
We consider another operation on a tree $T$: Take two adjacent edges $e_1=v_1v_2$ and $e_2=v_2v_3$, and name the component of $T\setminus\{e_1,e_2\}$ containing $v_i$ as $T_i$, for $i=1,2,3$. If $\vol_{T_1}(T_1)>\vol_{T_2}(T_2)$, we construct another tree $T'$ from $T$ by deleting the edge $e_2$ and adding an edge $e_2'=v_1v_3$ with weight $\omega(e_2'):=\omega(e_2)$. In this case, we say \emph{$T$ edge-transfers to $T'$ with respect to volume}.

\begin{lemma}\label{inequality_kemeny}
If $T$ edge-transfers to $T'$ with respect to volume, then $\kappa(T)>\kappa(T')$.
\end{lemma}
\begin{proof}
We keep the notations as in the previous paragraph. Similarly as in the proof of Lemma~\ref{inequality_aht}, we have $\tau(T)=\omega(T)=\omega(T')=\tau(T')=:\tau$ and $\vol_T(T)=\vol_{T'}(T')=:\vol$.

We have $V_T(T\setminus e)=V_{T'}(T'\setminus e)$ for any edge $e\not\in \{e_1,e_2,e_2'\}$. Indeed, denoting the component of $T\setminus e$ (resp. $T'\setminus e$) containing $v_1$ as $T_1$ (resp. $T_1'$), and the other component as $T_2$ (resp. $T'_2$), we must have $T_2=T_2'$ as graphs. See Figure~\ref{fig:proof} for an illustration. This implies
    \[\vol_{T}(T_2)=\vol_{T_2}(T_2)+\omega(e)=\vol_{T'_2}(T'_2)+\omega(e)=\vol_{T'}(T'_2),\]
and $\vol_{T}(T_1)=\vol -\vol_{T}(T_2)=\vol -\vol_{T'}(T'_2)=\vol_{T'}(T'_1)$. So $V_T(T\setminus e)=V_{T'}(T'\setminus e)$ as claimed. We also have $\omega(T\setminus e)=\tau/\omega(e)=\omega(T'\setminus e)$. A similar argument gives us $V_T(T\setminus e_2)=V_{T'}(T'\setminus e'_2)$ and $\omega(T\setminus e_2)=\tau/\omega(e_2)=\tau/\omega(e'_2)=\omega(T'\setminus e'_2)$. Therefore, the difference between $\kappa(T)$ and $\kappa(T')$ in the formula of Lemma~\ref{forestformula_kemeny} is only contributed by the $2$-forests $T\setminus e_1$ and $T'\setminus e_1$:\begin{equation}\label{eq:inequality_kemeny_1}
    \kappa(T)-\kappa(T')=\frac{1}{\vol \cdot \omega(e_1)}(V_T(T\setminus e_1)-V_{T'}(T'\setminus e_1)).
\end{equation} Here we used $\omega(T\setminus e_1)=\tau/\omega(e_1)=\omega(T'\setminus e_1)$.

Now we write $\omega_1:=\omega(e_1)$ and $\omega_2:=\omega(e_2)=\omega(e'_2)$. The two components of $T\setminus e_1$ are $T_1$ and $T_2\cup T_3$ respectively. Noticing that $d_T(v_1)=d_{T_1}(v_1)+\omega_1$ and $d_T(v)=d_{T_1}(v)$ for any $v\neq v_1$ in $T_1$, we have
   \[\vol_{T}(T_1)=\vol_{T_1}(T_1)+\omega_1.\]
Also notice that $d_{T}(v_2)=d_{T_2}(v_2)+\omega_1+\omega_2$, $d_{T}(v_3)=d_{T_3}(v_3)+\omega_2$, $d_T(v)=d_{T_2}(v)$ for any $v\neq v_2$ in $T_2$, and $d_T(v)=d_{T_3}(v)$ for any $v\neq v_3$ in $T_3$. So we have
    \[\vol_{T}(T_2\cup T_3)=\vol_{T_2}(T_2)+\omega_1+2\omega_2+\vol_{T_3}(T_3).\]
Similarly, the two components of $T'\setminus e_1$ are $T_1\cup T_3$ and $T_2$ respectively, and we can check $\vol_{T'}(T_1\cup T_3)=\vol_{T_1}(T_1)+\omega_1+2\omega_2+\vol_{T_3}(T_3)$ and $\vol_{T'}(T_2)=\vol_{T_2}(T_2)+\omega_1$. So we have the following computation.
\begin{align*}
    V_T(T\setminus e_1)-V_{T'}(T'\setminus e_1)&=(\vol_{T_1}(T_1)+\omega_1)(\vol_{T_2}(T_2)+\omega_1+2\omega_2+\vol_{T_3}(T_3))\\
    &\hspace{20pt}-(\vol_{T_1}(T_1)+\omega_1+2\omega_2+\vol_{T_3}(T_3))(\vol_{T_2}(T_2)+\omega_1)\\
    &=(\vol_{T_3}(T_3)+2\omega_2)(\vol_{T_1}(T_1)-\vol_{T_2}(T_2))>0,
\end{align*}since $\vol_{T_1}(T_1)>\vol_{T_2}(T_2)$ by our assumption of the edge-transfer. This inequality together with \eqref{eq:inequality_kemeny_1} concludes $\kappa(T)>\kappa(T')$.
\end{proof}

Let us write $T\succeq_v T'$ if $T'$ can be obtained after applying multiple (possibly zero times) edge-transfers with respect to volume starting with $T$. As a consequence of Lemma~\ref{inequality_kemeny}, the relation $\succeq_v$ defines a partial order on $\mathcal{T}_W$.

\begin{lemma}\label{poset_kemeny}
For any multiset $W$ of positive numbers, the maximal elements of the poset $(\mathcal{T}_W,\succeq_v)$ are the weighted paths, and the minimal element of $(\mathcal{T}_W,\succeq_v)$ is the weighted star $S_W$.
\end{lemma}
The proof of this lemma is very similar to that of Lemma~\ref{poset_aht}.
\begin{proof}
Suppose $T'\in \mathcal{T}_W$ is not a weighted path, then $T'$ has a vertex $v$ with at least three vertices $u_1,u_2,u_3$ adjacent to $v$. Let $T'_i$ be the component of $T'\setminus vu_i$ containing $u_i$, for $i=1,2,3$. Without loss of generality, we can assume $\vol_{T'_1}(T'_1)\geq \vol_{T'_2}(T'_2)$. Now, we construct another tree $T$ in $\mathcal{T}_W$ from $T'$ by deleting the edge $vu_3$ and adding the edge $u_2u_3$ with weight $\omega(u_2u_3):=\omega(vu_3)$. Then $T$ edge-transfers to $T'$ with respect to volume. Indeed, in $T$, if we name $v_1:=v$, $v_2:=u_2$, $v_3:=u_3$, $e_1:=v_1v_2$, $e_{2}:=v_2v_3$ and the component of $T\setminus\{e_1,e_2\}$ containing $v_i$ as $T_i$, we have $\vol_{T_1}(T_1)> \vol_{T'_1}(T'_1)\geq \vol_{T'_2}(T'_2)= \vol_{T_2}(T_2)$. So we can perform an edge-transfer on $T$, taking $e_1,e_2$ as the two adjacent edges, to obtain $T'$. By Lemma~\ref{inequality_aht}, $T$ and $T'$ are not isomorphic, so $T'$ is not a maximal element of $(\mathcal{T}_W,\succeq_s)$. On the other hand, since every vertex of a weighted path $P\in \mathcal{T}_W$ has at most two neighbors, $P$ cannot be a result of an edge-transfer, so $P$ is maximal with respect to $\succeq_s$.

Suppose $T\in \mathcal{T}_W$ is not $S_W$, then $T$ contains a path of length three which we can write as $v_0v_1v_2v_3$. Let $C_i$ be the component of $T\setminus v_1v_2$ containing $v_i$, for $i=1,2$. Without loss of generality, we can assume $\vol_{C_1}(C_1)\geq \vol_{C_2}(C_2)$. Then we can perform an edge-transfer with respect to size on $T$ to obtain another tree $T'$. Indeed, if we name $e_1:=v_1v_2$, $e_{2}:=v_2v_3$ and the component of $T\setminus\{e_1,e_2\}$ containing $v_i$ as $T_i$, we have $\vol_{T_1}(T_1)=\vol_{C_1}(C_1)\geq \vol_{C_2}(C_2)>\vol_{T_2}(T_2)$. By Lemma~\ref{inequality_aht}, $T$ and $T'$ are not isomorphic, so $T$ is not a minimal element of $(\mathcal{T}_W,\succeq_s)$. On the other hand, it is easy to check that $S_W$ is minimal with respect to $\succeq_s$.
\end{proof}

\begin{proof}[Proof of Theorem~\ref{minmax_kemeny}]
As a consequence of Lemma~\ref{inequality_kemeny} and Lemma~\ref{poset_kemeny}, the elements that maximize $\kappa$ are among the weighted paths in $\mathcal{T}_W$, and $S_W$ is the unique element that minimizes $\kappa$. Let $P=v_1v_2\dots v_n$, where $n=|W|+1$, be a weighted path and name the edges $e_i:=v_iv_{i+1}$ for $i=1,\dots,n-1$.
According to Lemma~\ref{forestformula_kemeny}, we can compute\begin{align*}
    \kappa(P)&=\frac{1}{\vol_P(P)\tau(P)}\sum_{i=1}^{n-1} V_P(P\setminus e_i)\omega(P\setminus e_i)\\
    &=\frac{1}{(\sum_{w\in W} 2w)(\prod_{w\in W} w)}\sum_{i=1}^{n-1}\left( (\sum_{j=1}^{i-1}2\omega(e_j)+\omega(e_i))(\omega(e_i)+\sum_{k=i+1}^{n-1}2\omega(e_k))\frac{\prod_{w\in W} w}{\omega(e_i)}\right)\\
    &=\frac{2}{\sum_{w\in W} w}\sum_{i=1}^{n-1}\sum_{j=1}^{i-1}\sum_{k=i+1}^{n-1}\frac{\omega(e_j)\omega(e_k)}{\omega(e_i)}+n-\frac{1}{2}.
\end{align*}Hence the maximization of $\kappa$ is equivalent to the maximization of $\sum_{i=1}^{n-1}\sum_{j=1}^{i-1}\sum_{k=i+1}^{n-1}\frac{\omega(e_j)\omega(e_k)}{\omega(e_i)}$ as claimed.
\end{proof}

\section{Remarks}\label{section_remark}

\noindent 1. Theorem~\ref{minmax_aht} and Theorem~\ref{minmax_kemeny} can be interpreted as extremal characterizations of the Laplacian spectra for weighted trees in $\mathcal{T}_W$. This is because $\alpha$ and $\kappa$ satisfy the following spectral formulas: For a connected weighted graph $G$ of size $n$, write the eigenvalues of its \emph{combinatorial Laplacian matrix} as $0=\lambda_0<\lambda_1\leq \dots\leq \lambda_{n-1}$, and the eigenvalues of its \emph{normalized Laplacian matrix} as $0=\mu_0<\mu_1\leq \dots\leq \mu_{n-1}$, then we have
   \[\frac{n}{\vol_G(G)}\alpha(G)=\sum_{i=1}^{n-1} \lambda_i^{-1}\text{\qquad and \qquad} \kappa(G)=\sum_{i=1}^{n-1} \mu_i^{-1}.\]
We refer our readers to \cite{chung2021forest} (see also \cite{chebotarev2007graph,chebotarev2020hitting,levene2002kemeny}) for definitions of Laplacian matrices and proofs of these identities.

\medskip

\noindent 2. The two partial orders $\succeq_s$ and $\succeq_v$ are identical for unweighted trees. Indeed, if $T$ is unweighted in the definitions of $\succeq_s$ and $\succeq_v$, we have $\vol_{T_1}(T_1)=2|T_1|-1$ and $\vol_{T_2}(T_2)=2|T_2|$ by handshaking lemma, so $|T_1|>|T_2|$ if and only if $\vol_{T_1}(T_1)>\vol_{T_2}(T_2)$. Hence, we can write $\succeq:=\succeq_s=\succeq_v$ if we only consider unweighted trees. Figure~\ref{fig:hasse} illustrates the Hasse diagram of $\succeq$ on unweighted trees of size $7$.

\begin{figure}[ht]
    \centering
    \includegraphics[scale=1.5]{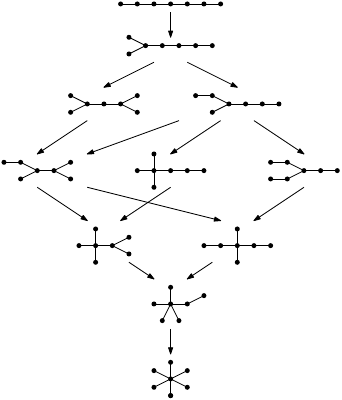}
    \caption{Hasse diagram of the partial order $\succeq$ on unweighted trees of size $7$.}
    \label{fig:hasse}
\end{figure}

Sidorenko \cite{sidorenko1994partially} introduced a partial order $\succcurlyeq$ before: For two unweighted trees $T$ and $T'$, we write $T \succcurlyeq T'$ if and only if $\text{hom}(T,G)\geq \text{hom}(T',G)$ for any unweighted graph $G$. Here, $\text{hom}(T,G)$ is the number of graph homomorphisms from $T$ into $G$. The Hasse diagrams of $\succcurlyeq$ on unweighted trees with sizes $6$ and $7$ are determined in \cite{sidorenko1994partially}. For unweighted trees of size at most $8$, we observed that there do not exist distinct trees $T$ and $T'$ satisfying both $T\succcurlyeq T'$ and $T\succeq T'$. In fact, for two trees $T$ and $T'$ with $|T|=|T'|\leq 8$, if $T\succcurlyeq T'$, then $\alpha(T')\geq \alpha(T)$. We wonder if these properties are true in general for trees of arbitrary sizes.

\medskip

\noindent {\bf Acknowledgement.} I wish to thank Fan Chung and Alexander Sidorenko for valuable discussions, and the anonymous referees of \emph{Discrete Applied Mathematics} for helpful comments.
\bibliographystyle{plain}
\bibliography{bib}
\end{document}